\documentclass{svjour3}   
\usepackage{amsxtra,amssymb,latexsym,amstext}
\usepackage{graphics, graphpap}
\usepackage[pdftex]{graphicx}

\usepackage{array, tabularx, longtable}

\usepackage{mathrsfs}
\usepackage{makeidx}

\smartqed  
\usepackage{hyperref}
\newcommand\R{{\mathbb R}}
\newcommand\E{{\textnormal E}}
\newcommand\PP{{\textnormal P}}

\begin{document}

\title{Conjunction probability of smooth centered Gaussian processes}
\subtitle{}

\titlerunning{Conjunction probability of smooth centered Gaussian processes}        

\author{Viet-Hung Pham
}


\institute{\at
              Institute of Mathematics, Vietnam Academy of Science and Technology (VAST)\\
              18 Hoang Quoc Viet, 10307 Hanoi, Vietnam\\
              \email{pgviethung@gmail.com, pvhung@math.ac.vn}                
}

\maketitle

\begin{abstract}
In this paper we provide an upper bound for the conjunction probability of independent Gaussian smooth processes and then we prove that this bound is a good approximation with exponentially smaller error. Our result confirms the heuristic approximation by Euler characteristic method of Worsley and Friston and also implies the exact value of generalized Pickands constant in a special case. Some results for conjunction probability of correlated processes are also discussed.


\end{abstract}
\section{Introduction}

In this paper, we investigate the conjunction probability of independent Gaussian processes, that is 

\begin{equation}\label{conju}
\PP \left(\underset{t\in [0,T]}{\sup} \underset{1\leq i \leq n}{\min} X_i(t) \geq u \right),
\end{equation}
where $u$ is a fixed threshold, and $X_i$'s are the independent smooth centered Gaussian processes with unit variance.

In a more general setting where $X_i$'s are random fields defined on $\R^d$, this problem has been addressed by Worsley and Friston in the seminal contribution \cite{MR2005444} with the statistical application to test whether the functional organization of the brain for language differs according to sex. With the same application to fRMI data, Alodat \cite{MR2775212} was interested in the distribution of the duration of the conjunction time. 

Most published papers \cite{MR3178344,MR3385594,MR2005444} assumed more that the processes $X_i$'s are stationary with the covariance functions $r_i(.), \, 1\leq i \leq n$ satisfying
$$r_i(t)=1-C_it^2+o(t^2)  \, \mbox{as} \, t\rightarrow 0, \; \mbox{and} \, r_i(t) <1, \, \forall t\in (0,T],$$
where $C_i$'s are positive constants. In this case, Debicki et al \cite{MR3178344} introduced the generalized Pickands constant 
$$H_{C_1,\ldots,C_n}=\underset{a\downarrow 0}{\lim} \frac{1}{a} \PP\left(\underset{k\geq 1}{\max} Z(ak) \leq 0\right),$$
where
$$Z(t)=\underset{1\leq i\leq n}{\min} \left( \sqrt{2}B_i(\sqrt{C_i}t)-C_i t^2+E_i\right),$$
with $B_i$'s being independent copies of a centered Gaussian process $B(t)$ with the covariance function $\textrm{Cov}(B(t),B(s))= |ts|,\; \forall t,s\geq 0$, and $E_i$'s being mutually independent unit mean exponential random variables and also independent of $B_i$'s. Using the double-sum method, they proved the asymptotic formula
$$\PP \left(\underset{t\in [0,T]}{\sup} \underset{1\leq i \leq n}{\min} X_i(t) >u \right)= H_{C_1,\ldots,C_n}T\frac{\varphi^n(u)}{u^{n-1}} (1+o(1)),$$
where $\varphi(.)$ is the density function of the standard normal distribution. However, one main disadvantage in statistical application of this result is the difficulty to estimate the exact value of the generalized Pickands constant $H_{C_1,\ldots,C_n}$.

 Worsley and Friston \cite{MR2005444} followed an heuristic argument that as the threshold $u$ is large enough, then the Euler characteristic $\chi(C_u)$ of the excursion set
$$C_u=\{t\in S:\; X_i(t) \geq u, \, \forall 1\leq i \leq n\}$$
just takes value 1 or 0 corresponding to the case $C_u$ is non-empty or empty. Under the same conditions on the stationary property of the processes, by Euler characteristic method, they considered the upper-triangular Toeplitz matrices $R_i$ corresponding to the process $X_i$ as
\begin{equation}\label{matrix}
R_i= \begin{pmatrix}
\overline{\Phi}(u) & \sqrt{C_i}\varphi(u)/\sqrt{2} \\
0 & \overline{\Phi}(u) 
\end{pmatrix},
\end{equation}
where  $\overline{\Phi}(.)$ is the tail distribution function of the standard normal distribution, and provided an heuristic and explicit approximation as

\begin{align}\label{ECH}
\PP \left(\underset{t\in [0,T]}{\sup} \underset{1\leq i \leq n}{\min} X_i(t) \geq u \right) & \approx \E (\chi(C_u))= (1,0)\left(\prod_{i=1}^n R_i\right)(1,T/\sqrt{\pi})^t\\
&=\overline{\Phi}^n(u)+\frac{\overline{\Phi}^{n-1}(u)\varphi(u)T}{\sqrt{2\pi}}\sum_{i=1}^n\sqrt{C_i}, \notag
\end{align}
where $(.)^t$ stands for the transpose of the vector.

However, they did not provide the validity of the above approximation. Therefore one does not know whether the approximation given by Euler characteristic method is nice or bad. It is worth to notice that the validity of Euler characteristic method is not obvious and trivial. For example, to study the tail distribution of the maximum of stationary Gaussian fields defined on the compact domain $S\subset \R^d$, this method is proven to be true for locally convex subset $S$ by Taylor, Akimichi and Adler \cite{MR2150192}; but it fails for non locally convex subsets (see \cite{MR3473099}). Note that once the validity is proven, then  the true value of the generalized Pickands constant is deduced immediately. 

In this paper, we will give an upper bound with two terms for the conjunction probability for every positive integer $n$. From the statistical point of view, a useful upper bound is better than an asymptotic formula. Furthermore, we will prove that our bound is sharp in the sense that the error is exponentially smaller.  As a consequence, our result confirms the validity Euler characteristic method and gives the explicit value of the generalized Pickands constants. The main theorem in this paper is stated as follows.

\begin{theorem}\label{thm1}
(a) Let $X_i, \, 1\leq i\leq n$ be $n$ independent centered Gaussian processes with continuously differentiable sample paths and unit variance. Then for any positive real number $u$,
$$\PP \left(\underset{t\in [0,T]}{\sup} \underset{1\leq i \leq n}{\min} X_i(t) \geq u \right) \leq \overline{\Phi}^n(u)+\frac{\overline{\Phi}^{n-1}(u)\varphi(u)}{\sqrt{2\pi}}\int_0^T \sum_{i=1}^n\sqrt{\mbox{Var}(X'_i(t))}dt.$$

(b) Assume more that for each $i=1,\ldots ,n$, the covariance function $r_i(s,t)$ is of class $\mathcal{C}^4$, that $|r_i(s,t)|<1$ for all $s\neq t$, and that $\left. \frac{\partial^2 r_i(s,t)}{\partial s \partial t}\right|_{s=t}=\mbox{Var}(X_i'(s))>0$ for all $s\in [0,T]$. Then there exists a positive constant $\delta$ such that
$$\PP \left(\underset{t\in [0,T]}{\sup} \underset{1\leq i \leq n}{\min} X_i(t) \geq u \right) = \overline{\Phi}^n(u)+\frac{\overline{\Phi}^{n-1}(u)\varphi(u)}{\sqrt{2\pi}}\int_0^T \sum_{i=1}^n\sqrt{\mbox{Var}(X'_i(t))}dt+O(\varphi(u(n+\delta))).$$
\end{theorem}
The main tool is the Rice formula to calculate the expectation of the number of "up-crossings". The detailed proof of the main theorem is presented in Section 2. In Section 3, we will apply the method to the conjunction probability of correlated processes.

\section{Proof of main theorem and discussions}

Before proving the main theorem, let us state some technical lemmas. The first lemma is a well-known result on the distribution of the maximum of Gaussian process (see \cite[Proposition 4.1]{MR2478201} or \cite{MR1361884}).
\begin{lemma}\label{le1}
(a) Let $\{X(t), t\in [0,T]\}$ be a centered Gaussian process with continuously differentiable sample paths and unit variance. Then for any positive real number $u$,
$$\PP \left(\underset{t\in [0,T]}{\max}  X(t) \geq u \right) \leq \overline{\Phi}(u)+\frac{\varphi(u)}{\sqrt{2\pi}}\int_0^T \sqrt{\mbox{Var}(X'(t))}dt.$$

(b) Assume more that the covariance function $r_X(s,t)$ is of class $\mathcal{C}^4$, that $|r_X(s,t)|<1$ for all $s\neq t$, and that $\mbox{Var}(X'(s))>0$ for all $s\in [0,T]$. Then there exists a positive constant $\delta$ such that
$$\PP(D_u>1)\leq  \E (D_u(D_u-1))/2 =O(\varphi(u(1+\delta))),$$
and
$$\PP(U_u>1)\leq  \E (U_u(U_u-1))/2 =O(\varphi(u(1+\delta))),$$
where $D_u$ ($U_u$) stands for the number of $u$-"down-crossings" ("up-crossings") as
$$D_u=\mbox{card}\{t\in (0,T): X(t)=u, \, X'(t) \leq 0\},$$
and
$$U_u=\mbox{card}\{t\in (0,T): X(t)=u, \, X'(t) \geq 0\},$$
\end{lemma} 
The second lemma states that there is no chance to see that both processes take the given values at a same point.
\begin{lemma} \label{le2}
Let $X_1(t)$ and $X_2(t)$ be two independent Gaussian processes with continuously differentiable sample paths. Then for a given $u$,
$$\PP(\exists t\in [0,T]:\, X_1(t)=X_2(t)=u)=0.$$
\end{lemma}
\begin{proof}
It is clear that for each positive $\epsilon$,
\begin{align*}
\PP(\exists t\in [0,T]:\, X_1(t)=X_2(t)=u)& \leq \PP(\exists t\in [0,T]:\, X_1(t)=u \, \mbox{and} \, |X_2(t)-u|\leq \epsilon)\\
&\leq  \E \left(\mbox{card}\{t\in [0,T]:\, X_1(t)=u \, \mbox{and} \, |X_2(t)-u|\leq \epsilon\}\right).
\end{align*}
By the Rice formula (see \cite{MR2478201}), the above expectation is equal to 
\begin{align*}
&\int_0^T \E(|X'_1(t)|\mathbb{I}_{\{|X_2(t)-u|\leq \epsilon\}}\mid X_1(t)=u)p_{X_1(t)}(u)dt\\
=&\PP(|X_2(t)-u|\leq \epsilon)\int_0^T \E(|X'_1(t)|\mid X_1(t)=u)p_{X_1(t)}(u)dt,
\end{align*}
where $p_{X_1(t)}(.)$ is the density function of the random variable $X_1(t)$. 

Let $\epsilon$ tend to 0, the result follows.
\end{proof}

\subsection{Proof of part (a): Upper bound}
  It is clear that
\begin{align*}
&\PP \left(\underset{t\in [0,T]}{\sup} \underset{1\leq i \leq n}{\min} X_i(t) \geq u\right)\\
=&\PP\left(X_i(0)\geq u, \forall i\right)+\PP\left(\{\exists i:\, X_i(0)<u\}\cap \left\{\underset{t\in [0,T]}{\sup} \underset{1\leq i \leq n}{\min} X_i(t) \geq u\right\}\right).
\end{align*}
Since the $n$-dimensional curve $(X_1(t),\ldots,X_n(t))$ is continuous, then under the condition $\{\exists i:\, X_i(0)<u\}\cap \left\{\underset{t\in [0,T]}{\sup} \underset{1\leq i \leq n}{\min} X_i(t) \geq u\right\}$ (it means that we start from a point outside and go inside the domain $\{(x_1,\ldots ,x_n) \in \R^n:\; x_i\geq u,  \forall i\}$ ), there exists at least one point $t\in [0,T]$ such that the curve touches the boundary of the domain, i.e.
$$\{\exists i\in\{1,\ldots,n\}: \; X_i(t)=u,X'_i(t)\geq 0, \,\mbox{and} \, X_j(t)\geq u ,\forall j\neq i\}. $$
Denote $U^*_u$ by the number of points satisfying the above condition. For each $i=1,\ldots,n$, denote $U_{i,u}$ by
$$U_{i,u}=\mbox{card}\{t\in [0,T]:\, X_i(t)=u,X'_i(t)\geq 0, \,\mbox{and} \, X_j(t)> u ,\forall j\neq i\}.$$
Thanks to Lemma \ref{le2},
\begin{equation}\label{bonfe}
\displaystyle \PP(U^*_u>0)=\PP(\underset{i=1}{\overset{n}{\cup}}\{U_{i,u}>0\}).
\end{equation}

 Then we have 
\begin{align*}
\PP \left(\underset{t\in [0,T]}{\sup} \underset{1\leq i \leq n}{\min} X_i(t) \geq u\right) & \leq  \overline{\Phi}^n(u)+\PP(U^*_u>0)\\
& \leq   \overline{\Phi}^n(u)+\sum_{i=1}^n\PP(U_{i,u}>0)\leq \overline{\Phi}^n(u)+\sum_{i=1}^n\E(U_{i,u})
\end{align*}
By the Rice formula, we have for each $i=1,\ldots,n$, 
\begin{align*}
E(U_{i,u})=& \int_0^T \E(\max\{X'_1(t),0\}  \prod_{j=1, j\neq i}^n\mathbb{I}_{\{X_j(t)>u\}}\mid X_i(t)=u)p_{X_i(t)}(u)dt\\
=& \prod_{j=1, j\neq i}^n\PP(X_j(t)>u) \int_0^T \E(\max\{X'_1(t),0\} )p_{X_i(t)}(u)dt\\
=& \frac{\overline{\Phi}^{n-1}(u)\varphi(u)}{\sqrt{2\pi}}\int_0^T \sqrt{\mbox{Var}(X'_i(t))}dt,
\end{align*}
here we use the fact that the processes $X_j$'s are independent and $X'_i(t)$ is independent of $X_i(t)$.

Summing up the expectations $E(U_{i,u})$'s, we have the upper bound
$$\PP \left(\underset{t\in [0,T]}{\sup} \underset{1\leq i \leq n}{\min} X_i(t) \geq u \right) \leq \overline{\Phi}^n(u)+\frac{\overline{\Phi}^{n-1}(u)\varphi(u)}{\sqrt{2\pi}}\int_0^T \sum_{i=1}^n\sqrt{\mbox{Var}(X'_i(t))}dt.$$

\subsection{Proof of part (b): Good approximation}
 To prove the sharpness of the upper bound, we first notice that
\begin{align*}
&\PP \left(\underset{t\in [0,T]}{\sup} \underset{1\leq i \leq n}{\min} X_i(t) \geq u\right)=\overline{\Phi}^n(u)+\PP(U^*_{u}>0)-\PP(\{X_j(0)>u, \forall j\} \cap \{U^*_u>0 \})\\
= & \overline{\Phi}^n(u)+\PP(\underset{i=1}{\overset{n}{\cup}}\{U_{i,u}>0\})-\PP(\{X_j(0)>u, \forall j\} \cap (\underset{i=1}{\overset{n}{\cup}}\{U_{i,u}>0\}))\\
\geq & \overline{\Phi}^n(u)+\sum_{i=1}^n\PP(U_{i,u}>0) - \sum_{i\neq j} \PP(\{U_{i,u}>0\}\cap \{U_{j,u}>0\})-\PP(X_j(0)>u, \forall j, \, U^*_u>0 )\\
\geq &  \overline{\Phi}^n(u)+\sum_{i=1}^n\left(\E(U_{i,u}) -\frac{\E[U_{i,u}(U_{i,u}-1)]}{2}\right) - \sum_{i\neq j} \PP(\{U_{i,u}>0\}\cap \{U_{j,u}>0\})\\
& \quad \quad \quad -\sum_{i=1}^n\PP(\{X_j(0)>u, \forall j\}\cap\{ U_{i,u}>0\} ),
\end{align*}
where the third line follows form Bonferroni inequality and the last line follows from the fact that the random variable $U_{i,u}$ takes integer values.

Then the result follows immediately if we can show that for  $i \neq j$, three terms $\E[U_{i,u}(U_{i,u}-1)]$, $\PP(\{U_{i,u}>0\}\cap \{U_{j,u}>0\})$ and $\PP(\{X_j(0)>u, \forall j\}\cap\{ U_{i,u}>0\} )$ are $O(\varphi(u(n+\delta)))$ for some positive $\delta$.

$\bullet$ For the first term, it is clear that
$$U_{i,u}(U_{i,u}-1)\leq U_i(U_i-1)\prod_{j=1,j\neq i}^n \mathbb{I}_{\left\{\underset{t\in [0,T]}{\max} X_j(t) >u\right\}},$$
where $U_i$ is the usual number of $u$-up-crossings with respect only to $X_i(t)$, i.e.
$$U_i=\mbox{card}\{t\in (0,T): X_i(t)=u, \, X_i'(t) \geq 0\}.$$
Thanks to Lemma \ref{le1}, we have
\begin{align*}
\E[U_{i,u}(U_{i,u}-1)] &\leq \E\left[ U_i(U_i-1)\prod_{j=1,j\neq i}^n \mathbb{I}_{\left\{\underset{t\in [0,T]}{\sup} X_j(t) >u\right\}}\right]\\
&= \E[U_{i}(U_{i}-1)] \prod_{j=1,j\neq i}^n \PP\left( \underset{t\in [0,T]}{\max} X_j(t) >u\right)\\
&\leq O(\varphi(u(1+\delta_i))) \prod_{j=1,j\neq i}^n  \left( \overline{\Phi}(u)+\frac{\varphi(u)}{\sqrt{2\pi}}\int_0^T \sqrt{\mbox{Var}(X'_j(t))}dt\right)\\
&=O(\varphi(u(n+\delta))),
\end{align*}
where $\delta_i$ is introduced as in Lemma \ref{le1}(b) and $\delta$ is a sufficiently small enough positive constant.

$\bullet$ For the third term, 
\begin{align*}
&\PP(X_j(0)>u, \forall j, \, U_{i,u}>0 ) \\
 \leq &\PP(X_j(0)>u, \forall j\neq i) \PP(X_i(0)>u,U_i>0)\\
\leq &\overline{\Phi}^{n-1}(u) [\PP(X_i(0)>u,X_i(T)>u)+\PP(X_i(0)>u,X_i(T)<u,U_i>0)]\\
\leq &\overline{\Phi}^{n-1}(u) [\PP(X_i(0)+X_i(T)>2u)+\PP(D_i>1)],
\end{align*}
where  $D_i$ is the usual number of $u$-down-crossings with respect only to $X_i(t)$, i.e.
$$D_i=\mbox{card}\{t\in (0,T): X_i(t)=u, \, X_i'(t) \leq 0\}.$$
Since $X_i(0)+X_i(T)$ is a centered Gaussian random variable with variance strictly less than $4$, then for some positive $\delta$,
$$\PP(X_i(0)+X_i(T)>2u)=O(\varphi(u(1+\delta))).$$
Again from Lemma \ref{le1}(b), we have
$$\PP(D_i>1)=O(\varphi(u(1+\delta_i))).$$
So we can do similarly as for the first term to obtain a negligible upper bound for the third term.

$\bullet$ For the second term, thanks to Lemma \ref{le2}, the probability that both $X_i(t)$ and $X_j(t)$ are equal to $u$ simultaneously at a common point $t$ is 0, then
\begin{align*}
 &\PP(\{U_{i,u}>0\}\cap \{U_{j,u}>0\}) \leq \PP \left( \prod_{k \neq i,j} \mathbb{I}_{\left\{\underset{t\in [0,T]}{\max} X_k(t) >u\right\}}\right) \times\\
  & \times \bigg[\PP(\exists t_i<t_j:\, X_i(t_i)=X_j(t_j)=u, X'_i(t_i)\geq 0, X'_j(t_j)\geq 0, X_i(t_j)>u, X_j(t_i)>u) \\\
 &+ \PP(\exists t_i>t_j:\, X_i(t_i)=X_j(t_j)=u, X'_i(t_i)\geq 0, X'_j(t_j)\geq 0, X_i(t_j)>u, X_j(t_i)>u) \bigg].\\
\end{align*}
We just deal with the first case $t_i<t_j$, the rest case is similar. In this case, we have
\begin{align*}
&\PP(\exists t_i<t_j:\, X_i(t_i)=X_j(t_j)=u, X'_i(t_i)\geq 0, X'_j(t_j)\geq 0, X_i(t_j)>u, X_j(t_i)>u)\\
\leq & \PP(\underset{t\in [0,T]}{\max} X_i(t) >u) \bigg[ \PP(\{X_j(0)<u\}\cap\{\exists t_i<t_j: X_j(t_i)>u,X_j(t_j)=u,X'_j(t_j)\geq0\})\\
&\quad\quad\quad\quad\quad\quad\quad+ \PP(\{X_j(T)<u\}\cap\{\exists t_i<t_j: X_j(t_i)>u,X_j(t_j)=u,X'_j(t_j)\geq0\})\\
&\quad\quad\quad\quad\quad\quad\quad+\PP(X_j(0)\geq u,X_j(T)\geq u) \bigg]\\
\leq &  \PP(\underset{t\in [0,T]}{\max} X_i(t) >u) \bigg[\PP(U_j>1)+\PP(D_j>1)+\PP(X_j(0)+X_j(T)\geq 2u)  \bigg],
\end{align*}
where $U_j$ (and $D_j$) is the usual number of $u$-upcrossings (downcrossings) with respect only to $X_j(t)$ as defined above. 

Then we can apply the same arguments as for two terms above and complete the proof of the main theorem.
\subsection{Discussions}

Remark that our result is general in the sense that we do not require the stationary assumption as in \cite{MR3385594,MR2005444}. Under this additional condition, we have the following corollary.
\begin{corollary} Let $X_i, \, 1\leq i\leq n$ be $n$ independent stationary centered Gaussian processes with continuously differentiable sample paths and covariance functions $r_i(.), \, 1\leq i \leq n$ that satisfy
$$r_i(t)=1-C_it^2+o(t^2)  \, \mbox{as} \, t\rightarrow 0, \; \mbox{and} \, r_i(t) <1, \, \forall t\in (0,T],$$
where $C_i$'s are positive constants. Then for any positive real number $u$,
$$\PP \left(\underset{t\in [0,T]}{\sup} \underset{1\leq i \leq n}{\min} X_i(t) \geq u \right) \leq \overline{\Phi}^n(u)+\frac{\overline{\Phi}^{n-1}(u)\varphi(u)T}{\sqrt{2\pi}} \sum_{i=1}^n\sqrt{C_i}.$$
Furthermore, there exists a positive constant $\delta$ such that
$$\PP \left(\underset{t\in [0,T]}{\sup} \underset{1\leq i \leq n}{\min} X_i(t) \geq u \right) = \overline{\Phi}^n(u)+\frac{\overline{\Phi}^{n-1}(u)\varphi(u)T}{\sqrt{2\pi}} \sum_{i=1}^n\sqrt{C_i}+O(\varphi(u(n+\delta))).$$
\end{corollary}
The proof follows from the fact that $\mbox{Var}(X'_i(t))=C_i$. It is clear that our bound coincides with the heuristic approximation given by the Euler characteristic method.  Furthermore, using the fact that for positive $u$,
$$\frac{\varphi(u)}{u}>\overline{\Phi}(u)>\frac{\varphi(u)}{u}-\frac{\varphi(u)}{u^3},$$
we deduce the explicit value of the generalized Pickands constant $H_{C_1,\ldots,C_n}$.
\begin{corollary} Under the stationary condition, we have
$$H_{C_1,\ldots,C_n}=\frac{1}{\sqrt{2\pi}} \sum_{i=1}^n\sqrt{C_i}.$$
\end{corollary}
\section{Conjunction probability of correlated processes}
In this section, we consider the conjunction probability of two correlated processes. To be precise, as in \cite[Section 5.4]{MR2206344}, let us consider $X,Y$ be independent copies of a stationary smooth centered Gaussian process with unit variance. Then for a fixed constant $\rho \in (-1,1)$, we define two correlated processes
\begin{equation}\label{corrl}
\begin{cases}
X_1=X\\
X_2=\rho X+\sqrt{1-\rho^2} Y,
\end{cases}
\end{equation}
and we are interested in the conjunction probability 
$$\PP \left(\underset{t\in [0,T]}{\max} \min \{ X_1(t), X_2(t)\} \geq u \right).$$
By the same method as for the independent processes, we can derive an upper bound. Unfortunately, we could not prove the sharpness of the given bound. We leave this question for future research.
\begin{theorem}
Let $X_1$ and $X_2$ be two correlated processes defined as in (\ref{corrl}). Then for every positive $u$,
\begin{align*}
\PP \left(\underset{t\in [0,T]}{\max} \min \{ X_1(t), X_2(t)\} \geq u \right)  \leq & 2\int_u^{\infty} \varphi(x) \overline{\Phi}\left( \sqrt{\frac{1-\rho}{1+\rho}}x\right)dx\\
& +2T\varphi(u) \frac{\sqrt{\mbox{Var}(X'_1(0))}}{\sqrt{2\pi}} \overline{\Phi}\left( \sqrt{\frac{1-\rho}{1+\rho}} u\right). 
\end{align*}

\end{theorem}
\begin{proof}
As in the proof of Theorem \ref{thm1}, we have the upper bound.
\begin{align*}
&\PP \left(\underset{t\in [0,T]}{\max} \min \{ X_1(t), X_2(t)\} \geq u \right)\\
\leq & \PP (X_1(0)\geq u, X_2(0)\geq u)+\PP (\exists t\in [0,T]:\, X_1(t)=u,X'_1(t)\geq 0, \,\mbox{and} \, X_2(t)> u )\\
& +\PP (\exists t\in [0,T]:\, X_2(t)=u,X'_2(t)\geq 0, \,\mbox{and} \, X_1(t)> u )\\
\leq & \PP (X_1(0)\geq u, X_2(0)\geq u)+\E (\mbox{card}\{ t\in [0,T]:\, X_1(t)=u,X'_1(t)\geq 0, \,\mbox{and} \, X_2(t)> u\} )\\
&+\E (\mbox{card}\{ t\in [0,T]:\, X_2(t)=u,X'_2(t)\geq 0, \,\mbox{and} \, X_1(t)> u\} ).
\end{align*}
It is easy to check that  (see also \cite[page 101]{MR2478201}) if 
$$(X_1(0),X_2(0)) \sim \mathcal{N}\left( 0, \begin{pmatrix}
1 & \rho \\
\rho & 1
\end{pmatrix} \right),$$
then
$$\PP (X_1(0)\geq u, X_2(0)\geq u) =2\int_u^{\infty} \varphi(x) \overline{\Phi}\left( \sqrt{\frac{1-\rho}{1+\rho}}x\right)dx.$$
By the Rice formula,
\begin{align*}
&\E (\mbox{card}\{ t\in [0,T]:\, X_1(t)=u,X'_1(t)\geq 0, \,\mbox{and} \, X_2(t)> u\} ) \\
=& \int_0^T \E(\max\{X_1'(t),0\}\mathbb{I}_{\{X_2(t)>u\}} \mid X_1(t)=u) p_{X_1(t)}dt\\
=& T \varphi(u) \E (\max\{X_1'(t),0\}) \PP(X_2(t)>u\mid X_1(t)=u)\\
=&T\varphi(u) \frac{\sqrt{\mbox{Var}(X'_1(0))}}{\sqrt{2\pi}}\PP(u\rho +\sqrt{1-\rho^2}Z>u)= T\varphi(u) \frac{\sqrt{\mbox{Var}(X'_1(0))}}{\sqrt{2\pi}} \overline{\Phi}\left( \sqrt{\frac{1-\rho}{1+\rho}} u\right),
\end{align*}
where in the second line, we use the stationary property and the fact that $X_1'(t)$ is independent of $X_1(t)$ and $X_2(t)$, and in the last line by the Gaussian regression of $X_2(t)$ under the condition $X_1(t)=u$, the random variable $Z$ has standard normal distribution.

The expectation $\E (\mbox{card}\{ t\in [0,T]:\, X_2(t)=u,X'_2(t)\geq 0, \,\mbox{and} \, X_1(t)> u\} )$ can be computed similarly. Taking the sum, we obtain the upper bound.
\end{proof}

\begin{acknowledgement} 
Part of this work was done during the author's post-doctoral fellowship of the Vietnam Institute for Advanced Study in Mathematics in 2016. This work is also funded by Vietnam National Foundation for Science and Technology Development (NAFOSTED) under grant number 101.03-2017.316.
\end{acknowledgement}

\end{document}